\documentclass{article}

\usepackage{amssymb}
\usepackage{latexsym}
\usepackage{amsthm}
\usepackage{enumerate}
\usepackage{epsfig}
\usepackage{color}
\usepackage{float}
\usepackage{subfigure}
\usepackage[alphabetic,nobysame]{amsrefs}
\usepackage{amsmath}
\usepackage{ifthen}
\usepackage{makeidx}
\usepackage{lscape}
\usepackage{amscd}
\usepackage{tikz}
\usetikzlibrary{calc}

\def\e{\e}

\def\H{\mathbb{H}}

\def\cross{\times}

\def\V{\mathcal{V}}
\def\W{\mathcal{W}}
\def\Xt{\widetilde{X}}
\def\e{\epsilon}

\def\ge{\geqslant}
\def\le{\leqslant}

\def\vol{\text{vol}}

\DeclareSymbolFont{AMSb}{U}{msb}{m}{n}
\DeclareMathSymbol{\Sph}{\mathbin}{AMSb}{"53}\DeclareMathSymbol{\Tor}{\mathbin}{AMSb}{"54}
\DeclareMathSymbol{\R}{\mathbin}{AMSb}{"52} \DeclareMathSymbol{\PP}{\mathbin}{AMSb}{"50}
\DeclareMathSymbol{\T}{\mathbin}{AMSb}{"54} \DeclareMathSymbol{\Z}{\mathbin}{AMSb}{"5A}
\DeclareMathSymbol{\C}{\mathbin}{AMSb}{"43} \DeclareMathSymbol{\E}{\mathbin}{AMSb}{"45}
\DeclareMathSymbol{\K}{\mathbin}{AMSb}{"4B}\DeclareMathSymbol{\N}{\mathbin}{AMSb}{"4E}

\def\T{{\cal{T}}}

\def\V{{\cal{V}}}

\newcommand{\norm}[1]{\left|#1\right|}
\newcommand{\Norm}[1]{\left|\left|#1\right|\right|}

\newtheorem{theorem}{Theorem}[section]

\newtheorem{proposition}[theorem]{Proposition}

\theoremstyle{definition}

\newtheorem{claim}[theorem]{Claim}

\newtheorem{definition}[theorem]{Definition}

\def\inj{\text{inj}}

\restylefloat{figure}

\begin{document}

\title{Morse area and Scharlemann-Thompson width for hyperbolic $3$-manifolds}
\author{Diane Hoffoss\footnote{email: dhoffoss@sandiego.edu} \, and Joseph Maher\footnote{email: joseph.maher@csi.cuny.edu}}
\date{\today}

\maketitle

\tableofcontents

\begin{abstract}
Scharlemann and Thompson define a numerical complexity for a
$3$-manifold using handle decompositions of the manifold.  We show
that for compact hyperbolic $3$-manifolds this is linearly related to
a definition of metric complexity in terms of the areas of level sets
of Morse functions.
\end{abstract}

\section{Introduction}

Let $M$ be a closed Riemannian $3$-manifold, and let $f \colon M \to
\R$ be a Morse function, i.e. $f$ is a smooth function, all of whose
critical points are non-degenerate, and for which distinct critical
points have distinct images in $\R$. We define the \emph{area} of $f$
to the maximum area of any level set $F_t = f^{-1}(t)$ over all points
$x \in \R$. We define the \emph{Morse area} of $M$ to be the infimum
of the area of all Morse functions $f \colon M \to \R$.

For hyperbolic $3$-manifolds, the hyperbolic metric is a topological
invariant by Mostow rigidity, and the critical points of a Morse
function determine a handle decomposition of the manifold, so one might
hope that Morse area is related to a topological measure of complexity
defined in terms of handle decompositions of the manifold. We show
that Morse area is linearly related to a definition of topological
complexity we call \emph{Scharlemann-Thompson width} or \emph{linear
  width}, and which we now describe.

For a closed (possibly disconnected) surface $S$, we define the
complexity, or genus, of $S$ to be the sums of the genera of each
connected component. For a compact (possibly disconnected) surface
with boundary, we define the genus of $S$ to be the genus of the
surface obtained by capping off all boundary curves with discs. We
shall write $\norm{\partial S}$ for the number of boundary components
of $S$.

A \emph{handlebody} is a compact $3$-manifold with boundary,
homeomorphic to the regular neighborhood of a graph in
$\mathbb{R}^3$. Up to homeomorphism, a handlebody is determined by the
genus $g$ of its boundary surface.  Every $3$-manifold $M$ has a
\emph{Heegaard splitting}, which is a decomposition of the manifold
into two handlebodies. This immediately gives a notion of complexity
for a $3$-manifold, called the \emph{Heegaard genus}, which is the
smallest genus of any Heegaard splitting of the $3$-manifold.

There is a refinement of this, due to Scharlemann and Thompson
\cite{st}, which we now describe. Let $S$ be a closed surface, which
need not be connected. A \emph{compression body} $C$ is a compact
$3$-manifold with boundary, constructed by attaching some number of
$2$-handles to one side $S \cross \{ 0 \}$ of $S \cross I$. We do not
require compression bodies to be connected. We shall refer to $S
\cross \{1\}$ as the \emph{top boundary} $\partial_+ C$ of the
compression body, and the other boundary components of $C$ as the
\emph{lower boundary} $\partial_- C$. The lower boundary may be
disconnected, even if $C$ is connected, and any $2$-sphere components
are capped off with $3$-balls.  In particular, if a maximal number of
non-parallel $2$-handles are attached, then the resulting compression
body is a handlebody, so a handlebody is a special case of a
compression body.  A \emph{generalized Heegaard splitting}, which we
shall call a \emph{linear splitting}, is a decomposition of a closed
$3$-manifold $M$ into a linearly ordered sequence of compression
bodies $C_1, \ldots C_{2n}$, such that the upper boundary of an odd
numbered compression body $C_{2i+1}$ is equal to the top boundary of
the compression body $C_{2i+2}$, and the lower boundary of $C_{2i+1}$
is equal to the lower boundary of the previous compression body
$C_{2i}$. For the even numbered compression bodies $C_{2i}$, the top
boundary is equal to the upper boundary of $C_{2i-1}$, and the lower
boundary is equal to the lower boundary of $C_{2i+1}$. In the case of
the first and last compression bodies $C_1$ and $C_{2n}$, the lower
boundaries are empty. Let $H_i$ be the sequence of surfaces consisting
of the upper boundaries of the compression bodies $C_{2i-1}$ and
$C_{2i}$. The complexity $c(H_i)$ of the surface $H_i$ is the genus of
$H_i$, i.e. the sum of the genera of each connected component, and the
\emph{width} of the linear splitting is the maximum value of $c(H_i)$
over all upper boundaries. The \emph{Scharlemann-Thompson width},
which we shall also refer to as the \emph{linear width}, of a
$3$-manifold $M$ is the minimum width over all possible linear
splittings. As a Heegaard splitting is a special case of a linear
splitting, the Heegaard genus of $M$ is an upper bound for the linear
width of $M$.

There is a refinement of linear width known as \emph{thin position},
which we discuss when we use it in Section \ref{section:flow}.

\subsection{Results}

In order to bound Morse area in terms of linear width we shall assume
the following result announced by Pitts and Rubinstein \cite{pr} (see
also Rubinstein \cite{rubinstein}).

\begin{theorem} \cites{pr,rubinstein} \label{conjecture:pr} %
Let $M$ be a Riemannian $3$-manifold with a strongly irreducible
Heegaard splitting. Then the Heegaard surface is isotopic to a
minimal surface, or to the boundary of a regular neighborhood of a
non-orientable minimal surface with a small tube attached vertically
in the I-bundle structure.
\end{theorem}

A full proof of this result has not yet appeared in the literature,
though recent progress has been made by Colding and De Lellis
\cite{cdl}, De Lellis and Pellandrini \cite{dlp}, and Ketover
\cite{ketover}.

We shall show:

\begin{theorem} \label{theorem:morse} %
There is a constant $K > 0$, such that for any closed hyperbolic
$3$-manifold,
\begin{equation} 
K ( \text{linear width}(M) ) \le \text{Morse area}(M) \le 4 \pi
(\text{linear width}(M)), 
\end{equation}
where the right hand bound holds assuming Theorem
\ref{conjecture:pr}.
\end{theorem}

Our methods are effective, and the constant $K$ may be estimated using
a bound on the Margulis constant for $\H^3$, though we omit the
details of this calculation, as our methods seem unlikely to give an
optimal constant.

\subsection{Outline}

In Section \ref{section:metric bounds topology}, we show how to bound
linear width in terms of Morse area. A bound on the Morse area of $M$
gives a Morse function $f \colon M \to \R$, with bounded area level
sets, but with no \emph{a priori} bound on the topological complexity
of the level sets.

We use a Voronoi decomposition of $M$ to give a polyhedral
approximation of the Morse function, which we now describe in a simple
case. Let $\V$ be a Voronoi decomposition of $M$ in which every
Voronoi cell $V_i$ is a topological ball, and has size bounded above
and below, i.e. there is an $\e > 0$ such that $B(x_i, \e/2) \subset
V_i \subset B(x_i, \e)$, where $x_i$ is the center of the Voronoi
cell. Let $M_t$ be the sublevel set of the Morse function, i.e. $M_t =
f^{-1}((-\infty, t])$. The sublevel sets are a monotonically
increasing collection of subsets of $M$, which start off empty, and
eventually contain all of $M$, so in particular, for each Voronoi cell
$V_i$ there is a $t_i$ such that the volume of $M_t \cap B(x_i, \e/2)$
is exactly half the volume of $B(x_i, \e/2)$, and we shall call $t_i$
the \emph{cell splitter} for the Voronoi cell $V_i$.  Furthermore, we
may assume that the $t_i$ are distinct for distinct Voronoi cells.
This gives a linear order to the Voronoi cells, and we wish to show
that constructing the manifold by adding the Voronoi cells in this
order gives a bounded linear width handle decomposition for $M$. Let
$P_t$ be the union of the Voronoi cells whose cell splitters $t_i$ are
at most $t$. Each Voronoi cell is a ball, with a bounded number of
faces, so adding a Voronoi cell corresponds to adding a bounded number
of handles. It remains to show that the boundary of each $P_t$ has
genus bounded in terms of the area of the level set $F_t =
f^{-1}(t)$. Let $V_i$ and $V_j$ be two adjacent Voronoi cells, with
$V_i$ contained in $P_t$ and $V_j$ outside $P_t$, so their common face
is a subset of $\partial P_t$. Consider the sequence of balls $B(x,
\e/2)$, as $x$ runs along the geodesic from $x_i$ to $x_j$. At least
half the volume of $B(x_i, \e/2)$ is contained in $M_t$, and at most
half the volume of $B(x_j, \e/2)$ is contained in $M_t$, so there is
an $x$ such that exactly half the volume of $B(x, \e/2)$ is contained
in $M_t$, and so there is a lower bound on the area of $F_t \cap B(x,
\e/2)$. Therefore, a bound on the area of $F_t$ gives a bound on the
number of faces of $\partial P_t$. As each face has a bounded number
of edges, this gives a bound on the genus of $\partial P_t$, and hence
a bound on the linear width of $M$, though this bound depends on $\e$.

In order to produce a bound which works for any compact hyperbolic
manifold $M$, we use the Margulis Lemma and the thick-thin
decomposition for hyperbolic manifolds. There is constant $\mu$,
called a Margulis constant, such that any compact hyperbolic manifold
may be decomposed into a thick part $X_\mu$, where each point has
injectivity radius greater than $\mu$, and a thin part, where each
point has injectivity radius at most $\mu$, and which is a disjoint
union of solid tori.  If we choose $\e$ sufficiently small, then we
may choose a Voronoi decomposition of the thick part in which each
Voronoi cell has size bounded above and below, and run the argument in
the previous paragraph to control the genus of $\partial P_t$ inside
the thick part. We do not control the complexity of $\partial P_t$ in
the thin part, but as each component of the thin part is a solid
torus, we may cap off $\partial P_t \cap X_\mu$ with surfaces parallel
to $P_t \cap \partial X_\mu$, while still obtaining bounds on the
genus.  In order to bound the number of handles corresponding to
adding a Voronoi cell, we use a result of Kobayashi and Rieck
\cite{kobayashi-rieck} which gives bounds on the topological
complexity of the intersection of a Voronoi cell with the thin part.

The key problem for the upper bound is that the techniques of Pitts
and Rubinstein use sweepouts, so although their minimax construction
produces a sweepout of bounded area, we do not know how to directly
replace a bounded area sweepout with a bounded area
foliation. However, the upper bound is obtained in recent work of
Colding and Gabai \cite{colding-gabai}, using work of Colding and
Minicozzi \cite{colding-minicozzi} on the mean curvature flow, and we
describe their results in Section \ref{section:flow}.

\subsection{Acknowledgements}

The authors would like to thank Dick Canary, David Futer, David Gabai,
Joel Hass, Daniel Ketover, Sadayoshi Kojima, Yair Minksy, Yo'av Rieck
and Dylan Thurston for helpful conversations, and the Tokyo Institute
of Technology for its hospitality and support. The second author was
supported by the Simons Foundation CGM 234477 and PSC CUNY
TRADB-45-17.

\section{Morse area bounds Scharlemann-Thompson width} \label{section:metric
  bounds topology}

In this section we show that we can bound the Scharlemann-Thompson
width of a hyperbolic manifold in terms of its Morse area.  We will
approximate level sets by surfaces which are unions of faces of
Voronoi cells, and we start by describing the properties of the Voronoi
decompositions that we will use.

\subsection{Voronoi cells} \label{section:voronoi}

We will approximate the level sets of $f$ by surfaces consisting of
faces of Voronoi cells.  We now describe in detail the Voronoi cell
decompositions we shall use, and their properties.

A \emph{polygon} in $\H^3$ is a compact convex subset of a hyperbolic
plane whose boundary consists of a finite number of geodesic
segments. A \emph{polyhedron} in $\H^3$ is a convex topological
$3$-ball in $\H^3$ whose boundary consists of a finite collection of
polygons. A \emph{polyhedral cell decomposition} of $\H^3$ is a cell
decomposition in which which every $3$-cell is a polyhedron, each
$2$-cell is a polygon, and the edges are all geodesic segments. We say
a cell decomposition of a hyperbolic manifold $M$ is \emph{polyhedral}
if its preimage in the universal cover gives a polyhedral cell
decomposition of $\H^3$.

Let $X = \{ x_i \}$ be a discrete collection of points in
$3$-dimensional hyperbolic space $\H^3$. The Voronoi cell $V_i$
determined by $x_i \in X$ consists of all points of $M$ which are
closer to $x_i$ than any other $x_j \in X$, i.e.
\[ V_i = \{ x \in \H^3 \mid d(x, x_i) \le d(x, x_j) \text{ for all }
x_j \in X \}. \]
We shall call $x_i$ the \emph{center} of the Voronoi cell $V_i$, and
we shall write $\V = \{ V_i \}$ for the collection of Voronoi cells
determined by $X$. Voronoi cells are convex sets in $\H^3$, and hence
topological balls.  The set of points equidistant from both $x_i$ and
$x_j$ is a totally geodesic hyperbolic plane in $\H^3$.  A \emph{face}
$F$ of the Voronoi decomposition consists of all points which lie in
two distinct Voronoi cells $V_i$ and $V_j$, so $F$ is contained in a
geodesic plane. An \emph{edge} $e$ of the Voronoi decomposition
consists of all points which lie in three distinct Voronoi cells $V_i,
V_j$ and $V_k$, which is a geodesic segment, and a \emph{vertex} $v$
is a point lying in four distinct Voronoi cells $V_i, V_j, V_k$ and
$V_l$. By general position, we may assume that all edges of the
Voronoi decomposition are contained in exactly three distinct faces,
the collection of vertices is a discrete set, and there are no points
which lie in more than four distinct Voronoi cells. We shall call such
a Voronoi decomposition a \emph{regular} Voronoi decomposition, and it
is a polyhedral decomposition of $\H^3$ if every cell is compact.  As
each edge is $3$-valent, and each vertex is $4$-valent, this implies
that the dual cell structure is a simplicial triangulation of $\H^3$,
which we shall refer to as the \emph{dual triangulation}. The dual
triangulation may be realised in $\H^3$ by choosing the vertices to be
the centers $x_i$ of the Voronoi cells and the edges to be geodesic
segments connecting the vertices, and we shall always assume that we
have done this. In this case the triangles and tetrahedra are geodesic
triangles and geodesic tetrahedra in $\H^3$.

Given a collection of points $X = \{ x_i \}$ in a hyperbolic
$3$-manifold $M$, let $\Xt$ be the pre-image of $X$ in the universal
cover of $M$, which is isometric to $\H^3$. We say a subset of $\H^3$
is \emph{equivariant} if it is preserved by the covering translations
determined by the quotient $M$. As $\Xt$ is equivariant, the
$k$-skeleton of the corresponding Voronoi cell decomposition $\V$ of
$\H^3$ is also equivariant, for $0 \le k \le 3$, as are the
$k$-skeletons of the dual triangulation.

We now show that the interior of each Voronoi cell $V$ is mapped down
homeomorphically by the covering projection.  Suppose $y$ is a point
in the interior of a Voronoi cell $V$ with center $x$, so $d(x, y) <
d(x', y)$, for any other $x' \in X$. Let $g$ be a covering
translation, which is an isometry, so $d(x, y) = d(gx, gy)$. As
covering translations act freely, this implies that $gy$ lies in the
interior of the Voronoi cell corresponding to $gx \not = x$. Therefore
$\text{interior}(V)$ has disjoint translates under the group of
covering translations, and so is mapped down homeomorphically into
$M$, though the covering projection may identify distinct faces of a
Voronoi cell under projection into $M$.

By abuse of notation, we shall refer to the resulting polyhedral
decomposition of $M$ as the Voronoi decomposition $\V$ of $M$. By
general position, we may assume that $\V$ is regular.  The dual
triangulation also projects down to a triangulation of $M$, which we
will also refer to as the dual triangulation, though this
triangulation may no longer be simplicial.

We say a collection $X = \{ x_i \}$ of points in $M$ is
\emph{$\epsilon$-separated} if the distance between any pair of points
is at least $\epsilon$, i.e. $d(x_i, x_j) \ge \epsilon$, for all $i
\not = j$.

\begin{definition}
Let $M$ be a compact hyperbolic $3$-manifold.  We say a Voronoi
decomposition $\V$ is $\e$-regular, if it is regular, and it arises
from a maximal collection of $\e$-separated points.
\end{definition}

We shall write $B(x, r)$ for the closed metric ball of radius $r$
about $x$ in $M$,
\[  B(x, r) = \{ y \in M \mid d(x, y) \le r \},   \]
which need not be a topological ball.  As the cells of an $\e$-regular
Voronoi decompositions are determined by a maximal collection of
$\epsilon$-separated points in $M$, each Voronoi cell is contained in
a metric ball of radius $\e$ about its center. Furthermore, as the
points $x_i$ are distance at least $\e$ apart, each Voronoi cell
contains a metric ball of radius $\e/2$ about its center, i.e.
\[ B(x_i, \epsilon/2 )  \subset V_i \subset B(x_i, \epsilon).     \]

One useful property of $\e$-regular Voronoi decompositions is that the
boundary of any union of Voronoi cells is an embedded surface, in fact
an embedded normal surface in the dual triangulation, as we now
describe.

A \emph{simple arc} in the boundary of a tetrahedron is a properly
embedded arc in a face of the tetrahedron with endpoints in distinct
edges. A \emph{triangle} in a tetrahedron is a properly embedded disc
whose boundary is a union of three simple arcs, and a
\emph{quadrilateral} is a properly embedded disc whose boundary is the
union of four simple arcs. A \emph{normal surface} in a triangulated
$3$-manifold is a surface that intersects each tetrahedron in a union
of normal triangles and quadrilaterals.

\begin{proposition}
Let $M$ be a compact hyperbolic manifold, and let $\V$ be an
$\e$-regular Voronoi decomposition. Let $P$ be a union of Voronoi
cells in $\V$, and let $S$ be the boundary of $P$. Then $S$ is an
embedded surface in $M$.
\end{proposition}

\begin{proof}
The collection of Voronoi cells $P$ intersects a tetrahedron $T$ in
the dual triangulation in a regular neighborhood of the vertices of
$T$. If a tetrahedron $T$ has one or three vertices corresponding to
Voronoi cells in $P$, then $S$ intersects $T$ in a single normal
triangle. If $T$ has exactly two vertices corresponding to Voronoi
cells in $P$, then $S$ intersects $T$ in a single normal
quadrilateral. Therefore $S$ consists of at most one triangle or
quadrilateral in each tetrahedron, and so is an embedded normal
surface.
\end{proof}

We shall write $\inj_M(x)$ for the injectivity radius of $M$ at $x$,
i.e. the radius of the largest embedded ball in $M$ centered at
$x$. We shall write $\inj(M)$ for the injectivity radius of $M$, which
is defined to be
\[ \inj(M) = \inf_{x \in M} \inj_M(x). \]

We shall say a Voronoi cell $V_i$ with center $x_i$ is a \emph{deep}
Voronoi cell if the injectivity radius at $x_i$ is at least $4\e$,
i.e. $\inj_M(x_i) \ge 4\e$, and in particular this implies that the
metric ball $B(x_i, 3\e)$ is a topological ball. We shall also call
centers, faces, edges and vertices of deep Voronoi cells deep.
We shall write $\W$ for the subset of $\V$ consisting of deep Voronoi
cells.

We now show that there are bounds, which only depend on $\epsilon$, on
the number of faces of a deep Voronoi cell, and the number of edges
and faces of a deep Voronoi cell.

\begin{proposition} \label{prop:bound} %
Let $M$ be a compact hyperbolic $3$-manifold with an $\e$-regular
Voronoi decomposition $\V$, and let $\W$ be the collection of deep
Voronoi cells.  Then there is a number $J$, which only depends on
$\e$, such that each deep Voronoi cell $W_i \in \W$ has at most $J$
faces, edges and vertices.
\end{proposition}

\begin{proof}
Let $W$ be a Voronoi cell with center $x$, and with faces $F_1, \ldots
, F_n$. Let $x_i$ be the center of the Voronoi cell $W_i$ adjacent to
the face $F_i$. As $W$ is interior, $W_i$ also lies in $\W$.

If two Voronoi cells share a common face, then the distance between
their centers is at most $2 \e$.  Therefore all of the centers of the
Voronoi cells corresponding to the faces of $V$ are contained in the
metric ball $B(x, 2\e)$. This implies that the the balls of radius
$\e/2$ around the $x_i$ are contained in the metric ball $B(x,
5\e/2)$. As the $B(x_i, \epsilon/2)$ are all disjoint, this implies
that the number of faces is at most
\[ J_1 = \frac{ \vol_{\mathbb{H}^3}( B(x, 5 \e/2) ) }{
  \vol_{\mathbb{H}^3}( B(x, \e / 2) ) }. \]
%
%
Note that $J_1$ is also an upper bound for the maximum number of edges
in any face of a Voronoi cell, because every edge of that face is
contained in another face in that cell. So the total number of edges
is at most $J_1^2$, and by the formula for Euler characteristic, the
number of vertices is at most $J_1^2 + C$. Therefore we may choose $J$
to be $J_1^2 + 2$.
\end{proof}

An similar volume bound argument to the one above proves the
following:

\begin{proposition} \label{prop:bound2} %
Let $M$ be a compact hyperbolic $3$-manifold with an $\e$-regular
Voronoi decomposition $\V$. Then there is a number $L$, which depends
only on $\e$, such that for any deep Voronoi center $x_i$, the number
of Voronoi centers contained in $B(x_i, 3\e)$ is at most $L$.
\end{proposition}


\subsection{Polyhedral surfaces} \label{section:linear}

We may choose a Morse function $f \colon M \to \R$, such that the
complexity of $f$ is within some small $\delta > 0$ of the infimum,
i.e.
\[ \text{area}(F_t) \le \text{Morse area}(M) + \delta ,  \]
for all $t \in \R$. We now describe how to use the Morse function $f$
to give a linear ordering to the Voronoi cells in $\V$.

\begin{definition}
Let $M$ be a compact hyperbolic $3$-manifold, and let $f \colon M
\rightarrow \R$ be a Morse function.  Given $t \in \R$ define the {\it
  sublevel set of $M$ at $t$}, which we shall denote $M_t$, to be the
subset of $M$ consisting of the union of all level sets $F_t$ with $t
\in ( -\infty, t]$, i.e.
\[ M_t = f^{-1}( (- \infty, t ]   ).   \]
\end{definition}

For $t$ sufficiently small, $M_t$ is the empty set, and for $t$
sufficiently large $M_t$ is equal to all of $M$.  The region $M_t$
varies continuously in $t$, and is monotonically increasing in $t$.

\begin{definition} \label{definition:splitter} %
Let $M$ be a compact hyperbolic $3$-manifold with an $\e$-regular
Voronoi decomposition $\V$.  Let $f \colon M \to \R$ be a Morse
function.  For each Voronoi cell $V_i$ with center $x_i$ there is a
unique $t_i \in \R$ such that the surface $F_{t_i}$ divides the metric
ball $B(x_i, \epsilon/2)$ exactly in half by volume, i.e. $\vol(M_t
\cap B(x_i, \e/2)) = \tfrac{1}{2} \vol(B(x_i, \e/2))$.  We call this
$t_i$ the {\it cell splitter} of $V_i$.
\end{definition}

\begin{definition} \label{definition:generic} %
We say that a Morse function $f \colon M \to \R$ is \emph{generic}
with respect to a Voronoi decomposition $\mathcal{V}$ if the cell
splitters for distinct Voronoi cells $V_i$ correspond to distinct
points $t_i \in \R$, and no cell splitter is also a critical point for
the Morse function. We say a point $t \in \R$ is \emph{generic} if it
is not a critical point for the Morse function, and is not a cell
splitter.
\end{definition}

We may assume that $f$ is generic by an arbitrarily small perturbation
of $f$, and we shall always assume that $f$ is generic from now on. 

\begin{definition}
Let $M$ be a compact hyperbolic $3$-manifold with an $\e$-regular
Voronoi decomposition $\V$, and let $f \colon M \rightarrow \R$ be a
generic Morse function.  Let $\V$ be the Voronoi decomposition ordered
by the order inherited from the cell splitters $t_i$.  Given $t \in
\R$, Let $M_t$ be the sublevel set of $M$ at $t$.  We define $P_t$,
the \emph{polyhedral approximation to $M_t$}, to be the union of the
Voronoi cells $V_i$ with $t_i \le t$, and call $S_t = \partial P_t$
the \emph{polyhedral surface} determined by $t \in \R$.
\end{definition}

The polyhedral surface $S_t$ is a union of faces of the Voronoi cells,
and so is a normal surface in the dual triangulation.  We shall write
$\Norm{S_t}$ for the number of Voronoi faces the polyhedral surface
$S_t$ contains. We shall write $\Norm{S_t \cap \W}$ for the number of
faces in the polyhedral surface $S_t \cap \W$, which may have
boundary.  A schematic picture of a polyhedral surface is given below
in Figure \ref{pic:Pt}.

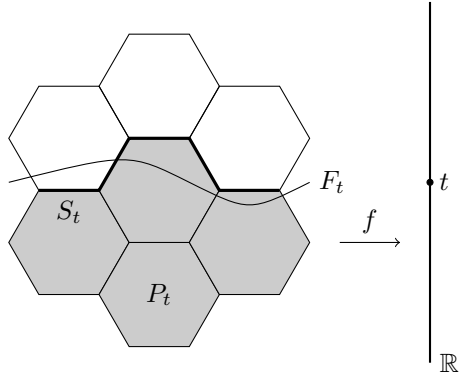
\begin{figure}[H]
\begin{center}
\begin{tikzpicture}[scale=0.4]

\def\hexagon{
\draw +(60:2) -- +(120:2) -- +(180:2) -- +(240:2) --
+(300:2) -- +(360:2) -- cycle;
}

\def\fillhexagon{
\filldraw [fill=black!20] +(60:2) -- +(120:2) -- +(180:2) -- +(240:2) --
+(300:2) -- +(360:2) -- cycle;
}

\fillhexagon

\begin{scope}[xshift=+0cm, yshift=+3.464cm]
    \hexagon
\end{scope}

\begin{scope}[xshift=+3cm, yshift=-1.732cm]
    \fillhexagon
\end{scope}

\begin{scope}[xshift=+3cm, yshift=+1.732cm]
    \fillhexagon
\end{scope}

\begin{scope}[xshift=+3cm, yshift=+5.196cm]
    \hexagon
\end{scope}

\begin{scope}[xshift=+6cm, yshift=0cm]
    \fillhexagon
\end{scope}

\begin{scope}[xshift=+6cm, yshift=3.464cm]
    \hexagon
\end{scope}

\draw [thick] (12,-4) node [right] {$\R$} -- (12,8);

\draw [->] (9, 0) -- (11,0) node [midway, above] {$f$};

\draw (-2, 2) .. controls (2,3) .. (4, 2)
              .. controls (6, 1) .. (8, 2) node [right] {$F_t$};

\filldraw (12, 2) circle (0.1cm) node [right] {$t$};

\draw [very thick] (-1,1.732) -- node [midway, below] {$S_t$}
++(0:2) -- ++(60:2) -- ++(0:2) -- ++(-60:2) -- ++(0:2);

\draw (3,-1.732) node {$P_t$};

\end{tikzpicture}%
\end{center} \caption{A polyhedral surface $S_t$ determined by a level
  set $F_t$.}  \label{pic:Pt}
\end{figure}

In this section, we will show the following bound on the complexity of
the polyhedral surface in the deep part $\W$.

\begin{proposition}\label{cor:bound} 
Let $M$ be a compact hyperbolic $3$-manifold, with an $\e$-regular
Voronoi decomposition $\V$, deep part $\W$, and a generic Morse
function $f \colon M \to \R$.  For $t \in \R$, let $S_t$ be the
polyhedral surface associated to $t$.  Then there is a constant $K$,
which only depends on $\e$, such that
\[ \norm{\partial ( S_t \cap \W ) } \le K \text{area}(F_t), \]
and
\[ \text{genus}(S_t \cap \W) \le K \text{area}(F_t). \]
\end{proposition}

In particular, this bounds the genus of $S_t \cap \W$ as constant
times the Morse width of $M$, where the constant depends only on $\e$.
We start by showing that the area the level sets bounds the number of
faces of the polyhedral surface in the deep part $\W$.

\begin{proposition} \label{bounded_faces} %
Let $M$ be a compact hyperbolic $3$-manifold, with an $\e$-regular
Voronoi decomposition $\V$, deep part $\W$, and a generic Morse
function $f \colon M \to \R$.  For $t \in \R$, let $S_t$ be the
polyhedral surface associated to $t$.  Then there is a constant $K$,
which only depends on $\e$, such that
\[ \Norm{S_t \cap \W} \le K \text{area}(F_t). \]
\end{proposition}

\begin{proof}
Let $P_t$ be the polyhedral approximation to $M_t$.  Let $C$ be a face
of $S_t \cap \W$, and let $W_i$ and $W_j$ be the two adjacent Voronoi
cells in $\V$. Up to relabeling we may assume that $W_i$ is
contained in $P_t$, and $W_j$ is not.  Let $\gamma$ be a geodesic
connecting $x_i$ to $x_j$, and consider $B(s, \epsilon/2)$, for $s \in
\gamma$. As $W_i$ and $W_j$ are deep, the metric balls $B(x_i,
\epsilon/2)$, $B(x_j, \epsilon/2)$ and $B(s, \epsilon/2)$ are all
topological balls, isometric to the ball $B(x,\e/2)$ in
$\mathbb{H}^3$.  At least half of the volume of $B(x_i, \epsilon/2)$
is contained in $P_t$, and strictly less than half of the volume of
$B(x_j, \epsilon/2)$ is contained in $P_t$, so there is some $s \in
\gamma$ such that exactly half the volume of $B(s, \epsilon/2)$ is
contained in $P_t$. There is a constant $A$, depending only on
$\epsilon$, such that any surface dividing a ball in hyperbolic into
regions of equal volume has area at least $A$.  In fact, we may take
$A$ to be the area of the equatorial disc, which is $2 \pi (
\cosh(\epsilon/2) - 1 )$; see for example Bachman, Cooper and White
\cite{bcw}.

Recall that the Voronoi decomposition has a dual triangulation in
which each edge is a geodesic segment, and we shall write $\Gamma$ for
the geodesic graph in $M$ formed by the $1$-skeleton of the dual
triangulation.  We shall write $\Gamma_d$ for the subset of $\Gamma$
consisting of vertices corresponding to deep Voronoi cells, and edges
connecting two deep Voronoi cells, and we shall refer to this as the
\emph{deep graph}.  Each geodesic edge between two deep Voronoi cells
has length strictly less than $2\epsilon$.  Therefore the choice of
geodesic is unique for the Voronoi cells in $\W$, as its length is
smaller than the injectivity radius at each deep Voronoi cell center
$x_i$. By Proposition \ref{prop:bound}, the geodesic dual graph
$\Gamma_d$ has valence at most $J$.

\begin{claim}\label{claim:area} 
Consider a collection of points $\{ s_i \}$ such that each point $s_i$
lies in a distinct edge $\gamma_i$ of the deep graph $\Gamma_d$. Then
any ball $B(s_i, \epsilon/2)$ intersects at most $L$ other balls
$B(s_j, \epsilon/2)$, where $L$ is the constant from Proposition
\ref{prop:bound2}.
\end{claim}

\begin{proof}[Proof of Claim \ref{claim:area}.]
If two balls $B(s_i, \e/2)$ and $B(s_j, \e/2)$ intersect, then the
distance between their corresponding edges $\gamma_i$ and 
$\gamma_j$ is at most $\e$, and so there is a pair of vertices, 
$x_k \in \gamma_i$ and $x_l \in \gamma_j$ with $d(x_k, x_l) \le 3 \e$. 
By Proposition \ref{prop:bound2}, there are at most $L$ other vertices 
within distance $3\e$ of a given vertex.  Therefore the total number 
of  balls intersecting $B(s_i, \e/2)$ is at most $L$, which only depends 
on $\e$.
\end{proof}

If there are $N$ faces in $S_t$ then there are at least $N/L$ disjoint
balls $B(s_i, \e/2)$, each containing a part of $F_t$ of area at least
$A$. Therefore, the total number of faces is at most
\begin{equation}\label{eq:linear surface} 
\Norm{S_t \cap \W} \le \frac{L}{A} \text{area}(F_t),
\end{equation}
where the constants only depend on $\e$, as required. 
\end{proof}

We now show that the bound on the number of faces of $S_t$ in the deep
part $\W$ gives a bound on the genus of $S_t \cap \W$.

\begin{proposition} \label{bounded_genus} %
Let $M$ be a compact hyperbolic $3$-manifold, with an $\e$-regular
Voronoi decomposition $\V$, deep part $\W$, and a generic Morse
function $f \colon M \to \R$.  For $t \in \R$, let $S_t$ be the
polyhedral surface associated to $t$.  Then there is a constant $J$,
which only depends on $\e$, such that:
\[ \norm{\partial ( S_t \cap \W )} \le J \Norm{S_t \cap \W}. \]
\[ \text{genus}(S_t \cap \W) \le J \Norm{S_t \cap \W}.  \]
Where $J$ is the constant from Proposition \ref{prop:bound}.
\end{proposition}

\begin{proof}
We shall write $S$ for $S_t$ to simplify notation.  The first bound
follows as each boundary component must contain at least one edge, so
the number of boundary components is at most the number of edges in $S
\cap \W$, which is at most $J\Norm{S \cap \W}$ by Proposition
\ref{prop:bound}.

We shall write $\widehat{S}$ for the surface $S \cap \W$ with all
boundary curves capped of with discs.  Recall that the genus of a
disconnected surface is the sum of the genera of each component, and
this in turn is equal to the number of connected components minus half
the Euler characteristic, i.e.
\[ \text{genus}(\widehat{S}) = | \widehat{S} | -
\tfrac{1}{2}\chi(\widehat{S}). \]
where $| \widehat{S} |$ is the number of connected components of
$\widehat{S}$.

As capping off with discs does not change the number of connected
components, this is at most the number of connected components of $S
\cap \W$, which is at most the number of faces $\Norm{S \cap \W}$.
Furthermore, capping off boundary components with discs may only
increase the Euler characteristic, so
\[ \text{genus}(\widehat{S}) \le \Norm{S \cap \W} - \tfrac{1}{2}
\chi(S \cap \W). \]
Therefore
\[ \text{genus}(\widehat{S}) \le \Norm{S \cap \W} - \tfrac{1}{2} (V -
E + F), \]
where $V, E$ and $F$ are the numbers of vertices, edges and faces of
$S \cap \W$. As each face of a deep Voronoi cell has at most $J$
edges, this implies
\[ \text{genus}(\widehat{S}) \le (1 + J/2)\Norm{S \cap \W}. \]
As we may assume that $J$ is at least $2$, this gives the second
inequality.
\end{proof}

Proposition \ref{cor:bound} now follows immediately from Propositions
\ref{bounded_faces} and \ref{bounded_genus}.

\subsection{Capped surfaces} \label{section:thickthin}


We have constructed surfaces with bounded complexity in the deep
part. The complement of the deep part is contained in a union of solid
tori by the Margulis Lemma, and we now explain how to cap off the
surfaces in the deep part with surfaces in the solid tori to produce
bounded genus surfaces.

We will use the Margulis Lemma and the \emph{thick-thin} decomposition
for finite volume hyperbolic $3$-manifolds, which we now review.
Given a number $\mu > 0$, let $X_\mu = M_{[\mu, \infty)}$ be the
\emph{thick part} of $M$, i.e. the union of all points $x$ of $M$ with
$\inj_M(x) \ge \mu$.  We shall refer to the closure of the complement
of the thick part as the \emph{thin part} and denote it by $T_\mu =
\overline{M \setminus X_\mu}$.

The Margulis Lemma states that there is a constant $\mu_0 > 0$, such
that for any compact hyperbolic $3$-manifold, the thin part is a
disjoint union of solid tori, and each of these solid tori is a
regular metric neighborhood of an embedded closed geodesic of length
less than $\mu_0$. We shall call a number $\mu_0$ for which this
result holds a \emph{Margulis constant} for $\mathbb{H}^3$. If $\mu_0$
is a Margulis constant for $\mathbb{H}^3$, then so is $\mu$ for any $0
< \mu < \mu_0$, and furthermore, given $\mu$ and $\mu_0$ there is a
number $\delta > 0$ such that $N_{\delta}(T_{\mu}) \subseteq
T_{\mu_0}$.  For the remainder for this section we shall fix a pair of
numbers $(\mu, \e)$ such that there are Margulis constants $0 < \mu_1
< \mu < \mu_2$, a number $\delta$ such that $N_{\delta}(T_{\mu})
\subseteq T_{\mu_2} \setminus T_{\mu_1}$, and $\e = \tfrac{1}{4} \min
\{ \mu_1, \delta \}$. We shall call $(\mu, \e)$ a choice of
\emph{MV}-constants for $\mathbb{H}^3$. This choice of constants
ensures that the deep part $\W$ is non-empty.

Let $(\mu, \e)$ be a choice of $MV$-constants, and consider an
$\e$-regular Voronoi decomposition of $M$. The fact that
$N_{\delta}(T_{\mu}) \subseteq T_{\mu_2} \setminus T_{\mu_1}$ means
that we adjust the boundary of $T_{\mu}$ by an arbitrarily small
isotopy so that it is transverse to the Voronoi cells, and we will
assume that we have done this for the remainder of this section.  Our
choice of $\e$ implies that the thick part $X_\mu$ is contained in the
Voronoi cells in the deep part, i.e. $X_\mu \subset \bigcup_{W_i \in
  \W} W_i$, so in particular $\partial X_\mu = \partial T_\mu$ is
contained in the deep part. Furthermore, as $\e < \delta$, each deep
Voronoi cell hits at most one component of $T_\mu$.

Each boundary component of the surface $S_t \cap X_\mu$ is contained
in $T_\mu$, so $S_t \cap X_\mu$ is a properly embedded surface in
$X_\mu$. We now bound the number of boundary components of $S_t \cap
X_\mu$ in terms of the number of polyhedral faces in the deep part,
$\Norm{S_t \cap \W}$.

\begin{proposition}\label{prop:thick bound}
Let $(\mu, \e)$ be $MV$-constants, and let $M$ be a compact hyperbolic
$3$-manifold with thin part $T_\mu$, an $\e$-regular Voronoi
decomposition $\V$ with deep part $\W$, and a generic Morse function
$f$.  Let $S_t$ be a polyhedral surface in $M$. Then there is a
constant $J$, depending only on $\e$, such that
\[ \text{genus}(S_t \cap X_\mu )  \le J \Norm{S_t \cap \W} \]
and
\[ \norm{ \partial (S_t \cap X_\mu ) } \le 2 J \Norm{S_t \cap \W}. \]
\end{proposition}

\begin{proof}
The properly embedded surface $S_t \cap X_\mu$ is obtained from $S_t
\cap \W$ by cutting $S_t \cap \W$ along simple closed curves and
discarding some connected components. This does not increase the
genus, which gives the first bound, using Proposition
\ref{bounded_genus}.

Each face $C$ of a Voronoi cell is a totally geodesic convex polygon,
and a component of $T_\mu$ lifts to a convex set in the universal
cover $\mathbb{H}^3$, so $C \cap T_\mu$ is a convex subset of
$C$. Therefore $C \cap \partial T_\mu$ consists either of a simple
closed curve, or a collection of properly embedded arcs which have at
most two endpoints in each edge of $C$, so there are at most as many
arcs as the number of edges of $C$. Therefore, the number of
components of $C \cap \partial T_\mu$ has at most (number of faces of
$S_t \cap \W$) plus (number of edges of $S_t \cap \W$) components, and
this gives the second bound, again using Proposition
\ref{bounded_genus}.
\end{proof}

We now wish to cap off the properly embedded surfaces $S_t \cap X_\mu$
with properly embedded surfaces in $T_\mu$ to form closed surfaces.
For each torus $T_i$ in $\partial T_\mu$ let $U_i$ be the subsurface
consisting of $\partial T_i \cap M_t$. Let $S_t^+ = (S \cap X_\mu)
\cup \bigcup_i U_i$, and we shall call the resulting closed surface
the \emph{$T$-capped surface} $S_t^+$. We now bound the genus of the
resulting $T$-capped surfaces.

\begin{proposition}\label{prop:bounded_genus}
Let $(\mu, \e)$ be $MV$-constants, and let $M$ be a compact hyperbolic
$3$-manifold with thin part $T_\mu$, an $\e$-regular Voronoi
decomposition $\V$ with deep part $\W$, and a generic Morse function
$f$.  Let $S_t$ be a polyhedral surface in $M$, and let $S_t^+$ be the
corresponding $T$-capped surface. Then there is a constant $K$,
depending only on $\e$, such that
\[ \text{genus}(S_t^+) \le K \text{area}(F_t). \]
Furthermore, for any finite collection of generic points $\{ u_i \}$
in $\R$, the corresponding $T$-capped surfaces $\{ S^+_{u_i} \}$ may
be isotoped to be disjoint.
\end{proposition}

\begin{proof}
By Proposition \ref{bounded_faces}, it suffices to bound the genus of the
$T$-capped surface in terms of the number of polyhedral faces of the
surface in the deep part. We will show
\[ \text{genus}(S_t^+) \le (5 J + 1) \Norm{S_t \cap \W}, \]
where $J$ is the constant from Proposition \ref{prop:bound}, which
only depends on $\e$.

Each surface $U_i$ is a subsurface of a torus, and so consists of a
union of planar surfaces, together with at most one surface which is a
torus with (possibly many) boundary components.

Capping off components of $S_t \cap X_\mu$ with planar surfaces cannot
increase the genus by more than twice the number of boundary
components, and capping off with punctured tori increases the genus by
at most the number of boundary components, plus the number of
punctured tori. As each Voronoi cell hits at most one component of
$T_\mu$, there are at most $\Norm{S_t \cap \W}$ components of the
$U_i$ surfaces which may be punctured tori. This implies
\[ \text{genus}(S_t^+) \le \text{genus}(S_t \cap X_\mu) + 2
\norm{ \partial (S_t \cap X_\mu) } + \Norm{S_t \cap \W}. \]
Using the bounds from Proposition \ref{prop:thick bound},
we obtain
\[ \text{genus}(S_t^+) \le (5 J + 1) \Norm{S_t \cap \W} \]
as required.

Finally, we show that for any finite collection of generic points $\{
u_i \}$ in $\R$, we may isotope the corresponding $T$-capped surfaces
to be disjoint. To simplify notation, given a generic point $u_i \in
\R$, we will write $M_i$ and $S_i$ for the corresponding polyhedral
approximation and polyhedral surface determined by $u_i$.

For any two distinct points $u_i < u_j$ in $\R$, the polyhedral
approximation $M_i$ is a strict subset of $M_j$, so $S_i$ and $S_j$
are disjoint normal surfaces.  Let $T$ be a single solid torus
component of $T_\mu$.  Take a small product neighborhood $\partial T
\cross [0, 1]$, and choose the parameterization such that $\partial T
\cross \{ 0 \}$ is equal to $\partial T$, and the product neighborhood
is contained in $T$.  Let $U_{i}$ be the subsurface of $\partial T$
given by $\partial T \cap M_{i}$. Let $U^+_i$ be the properly embedded
surface in the product $\partial T \cross [0, 1]$ given by placing
$U_i$ at depth $ i / n $, together with a product neighborhood of the
boundary $\partial U_i$ connecting $U_i$ to the boundary of $S_{i}$,
i.e.
\[ U^+_i = ( U_i \cross \{ i / n \} ) \cup ( \partial U_i \cross [0,
i/n] ).  \]
As the submanifolds $M_{i}$ are strictly nested, the subsurfaces
$U_i$ are also strictly nested, i.e. $U_i \subset U_j$ for $i < j$,
and so the resulting surfaces $S_{i} \cup U^+_i$ are disjoint.
\end{proof}

\subsection{Bounded handles}

We now bound the number of handles between a pair of $T$-capped
surfaces $S^+_i$ and $S^+_j$ whose corresponding points in $u_i$ and
$u_j$ in $\R$ bound an interval containing a single cell-splitter.

\begin{proposition}\label{bounded_handles} 
Let $(\mu, \e)$ be $MV$-constants, and let $M$ be a hyperbolic
$3$-manifold with an $\e$-regular Voronoi decomposition $\V$, and a
generic Morse function $f \colon M \to \R$. Let $u_1 < u_2$ be a pair
of points in $\R$, which bound an interval containing a single cell
splitter $t$. Let $S^+_1$ and $S^+_2$ be $T$-capped surfaces
corresponding to the level sets for $u_1$ and $u_2$, bounding regions
$P_1$ and $P_2$, with $P_1 \subset P_2$. Then $P_2$ is homeomorphic
to a manifold obtained from $P_1$ by adding at most $60 J^2 \max \{
\Norm{S_i \cap \W} \}$ handles, where $J$ is the constant from
Proposition \ref{prop:bound}, which depend only on $\e$.
\end{proposition}

We start with the observation that attaching a compression body $P$ to
a $3$-manifold $Q$ by a subsurface $S$ of the upper boundary component
of $P$, requires a number of handles which is bounded in terms of the
Heegaard genus of $P$, and the number of boundary components of the
attaching surface.

\begin{proposition}\label{bounded_handles3} 
Let $Q$ be a compact $3$-manifold with boundary, and let $R = Q \cup
P$, where $P$ is a compression body of genus $g$, attached to $Q$ by a
homeomorphism along a (possibly disconnected) subsurface $S$ contained
in the upper boundary component of $P$ of genus $g$. Then $R$ is
homeomorphic to a $3$-manifold obtained from $Q$ by the addition of at
most $(4 \text{genus}(P) + 2 \norm{\partial S})$ $1$-and $2$-handles,
where $\norm{\partial S}$ is the number of boundary components of $S$.
\end{proposition}

\begin{proof}
Recall that the genus of a disconnected surface with boundary is the
sum of the genus of each closed component obtained by capping off all
boundary components with discs. Therefore, the genus of $S$ is at most
the genus of $P$. For a connected surface of genus $g$ with $b$
boundary components, cutting along a non-separating arc with endpoints
in the same boundary component produces a surface of genus $g-1$ with
$b+1$ boundary components. A planar surface with $b$ boundary
components may be cut in to at most $b$ discs by $b-1$ non-separating
arcs. Therefore we may choose at most $2g + b$ arcs which cut the
surface $S$ into at most $g+b$ discs. We can add a $1$-handle to $Q$
for each arc, and then a $2$-handle for each disc, to produce a
manifold $Q_+$ which is homeomorphic to $Q$ union a regular
neighborhood of $\partial P$. We may then form $R$ by adding at most
$g$ $2$-handles. The total number of $1$- and $2$-handles required is
at most $4g+2b$.
\end{proof}

\begin{proof}[Proof (of Proposition \ref{bounded_handles}).]
Let $V$ be the Voronoi cell corresponding to the single cell splitter
$t$ contained in the interval $[u_1, u_2]$. The surfaces $S^+_1$ and
$S^+_2$ are parallel everywhere, except in a regular neighborhood of
$V$. If the Voronoi cell $V$ is disjoint from $T_\mu$, then it is a
ball, and is attached to $P_1$ along a subsurface consisting of a
union of faces of $V$. Therefore the number of boundary components of
the attaching surface is at most $J$, where $J$ is the constant from
Proposition \ref{prop:bound}, so by Proposition
\ref{bounded_handles3}, $P_2$ is obtained from $P_1$ by attaching at
most $2 J$ handles.

If the Voronoi cell $V$ intersects $T_\mu$, then $P_2$ is obtained
from $P_1$ by adding regions of $V \setminus T_\mu$, which we shall
refer to as the \emph{complementary regions}, together with regions of
$T_\mu \cap ( P_2 \setminus P_1 )$. The complementary regions may not
be topological balls, but Kobayashi and Rieck \cite{kobayashi-rieck}
show that they are handlebodies of bounded genus.

\begin{proposition}\cite{kobayashi-rieck}
Let $\mu$ be a Margulis constant for $\mathbb{H}^3$, $M$ be a finite
volume hyperbolic $3$-manifold, let $0 < \e < \mu$, and let $\V$ be a
regular Voronoi decomposition of $M$ arising from a maximal collection
of $\e$-separated points. Then there is a number $G$, depending only
on $\mu$ and $\e$, such that for any Voronoi cell $V_i$, there are at
most $G$ connected components of $V_i \cap X_\mu$, each of which is a
handlebody of genus at most $G$, attached to $T_\mu$ by a surface with
at most $G$ boundary components.
\end{proposition}

We state a simplified version of their result which suffices for our
purposes. Their stated result involves extra parameters $d$ and $R$,
but if $d$ is chosen close to $0$, then $R$ is close to $\mu$, and we
obtain the result above. Their proof involves showing that in the
universal cover, for any point $p$ in $T_\mu \cap V_i$, projection to
$\partial T_\mu$ along geodesic rays based at $p$ gives a topological
product structure to $V_i \cap X_\mu$ as $(V_i \cap \partial T_\mu)
\cross I$.  An examination of their proof shows that we may choose $G
= 3J$, where $J$ is the constant from Proposition \ref{prop:bound}.
Then Proposition \ref{bounded_handles3} implies that adding the
complementary regions of a Voronoi cell which intersects $\partial
T_\mu$ requires at most $6 G^2 = 54 J^2$ handles.

However, if the Voronoi cell intersects a solid torus component $T$ of
$T_\mu$, then the surfaces $S^+_1$ and $S^+_2$ need not be parallel
inside $T$, and so we now bound the number of handles needed to add
the region corresponding to $(P_2 \setminus P_1) \cap T$.  If $U_2$ is
equal to all of $\partial T$, then the additional region is a solid
torus attached along $\partial T \setminus U_1$, so adding this region
requires at most $4 + 2 \norm{\partial U_1}$ handles, by Proposition
\ref{bounded_handles}.  If $U_2$ is not equal to all of $\partial T$,
then this region is homeomorphic to $(U_2 \cross [0, 1]) \setminus
(U_1 \cross [0, \tfrac{1}{2}])$, and so is homeomorphic to $U_2 \cross
I$, which is a handlebody of genus at most $\norm{\partial U_2}$. The
region is attached along $U_2$, so adding this region requires at most
$4 \norm{\partial U_2} + 2 \norm{\partial U_1}$ handles, and so in
either case, at most $4J \Norm{S_2 \cap \W} + 2J \Norm{S_1 \cap \W}$
are required.

Therefore $P_2$ may be constructed from $P_1$ by adding at most 
\[ 54 J^2 + 4J \Norm{S_2 \cap \W} + 2J \Norm{S_1 \cap \W} \le 60 J^2
\max \{ \Norm{ S_i \cap \W } \} \]
handles, as required.
\end{proof}

The manifold $M$ may be constructed by adding the Voronoi cells in the
order arising from the cell splitters $t_i$ in $\R$. Choose a finite
collection of generic points $\{ u_i \}$, so that each pair of
adjacent cell splitters is separated by one of the $u_i$, and let $\{
S^+_i \}$ be the corresponding collection of $T$-capped surfaces.  The
linear width is at most the largest genus of any surface in the
collection $\{ S^+_{i} \}$, plus the maximum number of handles added
by attaching a single Voronoi cell. Therefore the bounds from
Propositions \ref{prop:bounded_genus} and \ref{bounded_handles} imply
\[ \text{linear width}(M) \le (5J + 1) K \left( \text{Morse area}(M)
\right) + 60 J^2 K \left( \text{Morse area}(M) \right) .\]
As $J$ is at least $1$, this gives
\[ \text{linear width}(M) \le \left( 66 J^2 K \right) \text{Morse
  width}(M).  \]
The constants $J$ and $K$ only depend on the choice of MV-constants,
which may be chosen independently of the hyperbolic $3$-manifold $M$,
and so this completes the proof of the left hand bound of Theorem
\ref{theorem:morse}.

\section{Scharlemann-Thompson width bounds Morse area}\label{section:flow}

In this section we will show linear bounds for Morse area in terms of
Scharlemann-Thompson width, assuming the Pitts and Rubinstein result
Theorem \ref{conjecture:pr}, i.e. we will show the right hand bound of
Theorem \ref{theorem:morse}. This result is due to Gabai and Colding
\cite{colding-gabai}*{Appendix A}, using recent work of Colding and
Minicozzi \cite{colding-minicozzi}, but we give a brief description
for the convenience of the reader, as they do not state this result
explicitly.

We will use properties of a refinement of linear width, known as
\emph{thin position}, which we now describe.  Let $\{ H_i \}$ be the
collection of upper boundaries of compression bodies in the linear
splitting, and let $c(H_i)$ be the complexity of the surface $H_i$,
i.e. the sum of the genera of its connected components. We say that
the complexity of the linear splitting is the collection of integers
$\{c(H_i)\}$, arranged in decreasing order.  A linear splitting which
gives the minimum complexity of all possible linear splittings in the
lexicographic ordering on sets of integers is called a \emph{thin
  position} linear splitting.  Scharlemann and Thompson \cite{st}
showed that thin position linear splittings have the following
property.

\begin{theorem} \cite{st} %
Let $H$ be a linear splitting that is in thin position. Then every
even surface is incompressible in $M$ and the odd surfaces form
strongly irreducible Heegaard surfaces for the components of $M$ cut
along the even surfaces.
\end{theorem}

If follows from Freedman, Hass and Scott \cite{fhs} that the
incompressible surfaces may be chosen to be disjoint least area
minimal surfaces, and in fact the off surfaces may also be chosen to
be disjoint minimal surfaces, see for example Lackenby \cite{lackenby}
or Renard \cite{renard}. In a hyperbolic manifold the intrinsic
curvature of a minimal surface is at most $-1$, so the Gauss-Bonnet
formula gives an upper bound for the area of the minimal
surface. Therefore the area of a minimal surface of genus $g$ is at
most $-2 \pi \chi(S) \le 4 \pi g$.

We say that a hyperbolic $3$-manifold $M$ has \emph{least area
  boundary} if its boundary components are (possibly empty) least area
minimal surfaces, and we say that a Heegaard splitting $H$ for $M$ is
\emph{minimal} if it is isotopic to an unstable minimal surface.  The
right hand bound of Theorem \ref{theorem:morse} is a consequence of
the following result of Colding and Gabai \cite{colding-gabai}, which
constructs bounded area foliations for a pair of compressions bodies
with least area lower boundaries, sharing a common minimal Heegaard
splitting surface.

\begin{theorem}\cite{colding-gabai}\label{theorem:minimal}
Let $M$ be a hyperbolic manifold, with (possibly empty) least area
boundary, with a minimal Heegaard splitting $H$ of genus $g$. Then,
assuming Theorem \ref{conjecture:pr}, the manifold $M$ has a (possibly
singular) foliation by compact leaves, containing the boundary
surfaces as leaves, such that each leaf has area at most $4 \pi g$.
\end{theorem}

As they do not state this explicitly in their paper, we give a brief
outline for the convenience of the reader.

\begin{definition} 
A mean convex foliation on a Riemannian $3$-manifold with boundary is
a smooth codimension-$1$ foliation, possibly with singularities of
standard type, such that each leaf is mean convex.
\end{definition}

In a $3$-manifold a foliation with singularities of ``standard type''
means that almost all leaves are completely smooth (i.e., without any
singularities). In particular, any connected subset of the singular
set is completely contained in a leaf. Furthermore, the entire
singular set is contained in finitely many (compact) embedded
Lipschitz curves with cylinder singularities together with a countable
set of spherical singularities.

The following result is shown by Colding and Gabai
\cite{colding-gabai}.

\begin{theorem} \cite{colding-gabai}*{Appendix A} %
Let $\Sigma$ be an unstable minimal surface in a hyperbolic manifold
$M$. Then there is a regular neighborhood of $\Sigma$ with a smooth
mean convex product foliation $\Sigma_t$, $t \in [-\e, \e]$, with
non-minimal boundary leaves $\Sigma_{-\e}$ and $\Sigma_{\e}$.
\end{theorem}

In particular, each leaf in the foliation has area at most $4 \pi g$.
As the boundary leaves $\Sigma_{-\e}$ and $\Sigma_{\e}$ are
non-minimal mean convex surfaces, we may apply the mean curvature flow
results of Colding and Minicozzi \cite{colding-minicozzi}, which show
that the mean curvature flow gives rise to a mean convex foliation
with standard singularities. As the mean curvature flow gives a
foliation by surfaces of decreasing area, the only possible
singularities which may arise are disc compressions, $2$-spheres
collapsing to a point or tori collapsing to circles. In particular,
each non-singular leaf bounds a compression body in the interior of
the compression body it is contained in.

If all leaves eventually collapse, then the compression body has empty
lower boundary, i.e. it is a handlebody, and this gives a mean convex
foliation, and hence area decreasing foliation, of the
handlebody. Otherwise, the mean curvature flow limits to a stable
minimal surface $\Gamma$ whose components bound compression bodies
together with the lower boundary of the original compression body.

If the stable minimal surface $\Gamma$ is not equal to the stable
boundary of the compression body, then it bounds a sub-compression
body with stable boundary, whose standard Heegaard splitting is
strongly irreducible, so we may apply the argument again. Anderson
\cite{anderson} and White \cite{white} showed that there are only
finitely many minimal surfaces of bounded genus in a compact
Riemannian manifold, and so this process may occur only finitely many
times, resulting in a foliation of the entire compression body. This
completes the proof of Theorem \ref{theorem:minimal}.

Finally we deduce the right hand bound of Theorem \ref{theorem:morse}
from Theorem \ref{theorem:minimal}.

\begin{proof}[Proof of right hand bound of Theorem
\ref{theorem:morse}.] By Conjecture \ref{conjecture:pr}, the
irreducible Heegaard surface for a hyperbolic $3$-manifold $M$ with
stable boundary is either isotopic to an unstable minimal surface
$\Sigma$, to which we may apply Theorem \ref{theorem:minimal}
directly, or isotopic to a regular neighborhood of a one-sided stable
minimal surface union a small tube parallel to one of the normal
fibers. In the latter case, the Heegaard surface bounds a handlebody
on at least one side, and cutting along the stable one-sided surface
leaves a compression body homeomorphic to the Heegaard surface cut
along the disc corresponding to the tube, where all boundary
components are stable minimal surfaces. As the standard Heegaard
splitting of a compression body is strongly irreducible, we may now
apply Theorem \ref{theorem:minimal} in this case as well.
\end{proof}



\bigskip

\noindent Diane Hoffoss \\
University of San Diego \\
\url{dhoffoss@sandiego.edu} \\

\noindent Joseph Maher \\
CUNY College of Staten Island and CUNY Graduate Center \\
\url{joseph.maher@csi.cuny.edu} \\


\end{document}